\documentclass[a4paper,11pt,reqno]{amsart}

\usepackage{graphicx}
\usepackage{a4wide}
\usepackage{eucal}
\usepackage[pdftex,colorlinks]{hyperref}

\numberwithin{equation}{section}

\newcommand{\R}{\mathbb R}
\newcommand{\C}{\mathbb C}

\newcommand{\be}{\begin{equation}}
\newcommand{\ee}{\end{equation}}
\newcommand{\ba}{\begin{eqnarray}}
\newcommand{\ea}{\end{eqnarray}}

\def\R{\mathbb{R}}

\def\C{\mathbb{C}}

\def\Cc{\mathcal{C}}

\def\Fc{\mathcal{F}}

\def\Dc{\mathcal{D}}
\def\Lc{\mathcal{L}}

\def\bu{\bar{u}}
\def\bv{\bar{v}}

\def\tU{\tilde{U}}
\def\tV{\tilde{V}}
\def\tC{\tilde{C}}
\def\tL{\tilde{\Lc}}

\def\Lip{\mathrm{Lip}}
\newtheorem{rema}{Remark}
\newtheorem{defi}{Definition}
\newtheorem{prop}{Proposition}
\newtheorem{coro}{Corollary}
\newtheorem{theo}{Theorem}
\newtheorem{lemm}{Lemma}

\begin{document}

\title[Stabilization of quasilinear hyperbolic systems of balance laws]{Boundary stabilization of quasilinear hyperbolic systems of balance laws: 
Exponential decay for small source terms}



\author[Gugat]{Martin Gugat}
\address{Friedrich-Alexander Universit\"at Erlangen-N\"urnberg (FAU), Department Mathematik, Cauerstr. 11, 91058 Erlangen, Germany}
\email{martin.gugat@fau.de}
\author[Perrollaz]{Vincent Perrollaz}
\address{Laboratoire de Math\'ematiques et Physique Th\'eorique, Universit\'e de Tours, UFR Sciences et Techniques, Parc de Grandmont, 37200 Tours, France}
\email{Vincent.Perrollaz@lmpt.univ-tours.fr}
\author[Rosier]{Lionel Rosier}
\address{Centre Automatique et Syst\`emes (CAS) and Centre de Robotique (CAOR), MINES ParisTech, PSL Research University, 60 Boulevard Saint-Michel, 75272 Paris Cedex 06, France}
\email{Lionel.Rosier@mines-paristech.fr}

\begin{abstract}
We investigate the long-time behavior of solutions of quasilinear hyperbolic systems with transparent boundary conditions when small source terms are 
incorporated in the system. Even if the finite-time stability of the system is not preserved, it is shown here that an exponential convergence towards
 the steady state still holds with a decay rate which is proportional to the logarithm 
of the amplitude of the source term. The result is stated for a system with dynamical boundary conditions in order to deal with initial data that are free of any compatibility condition.  
\end{abstract}



\noindent

%






\maketitle


\section {Introduction}
Solutions of certain hyperbolic systems can reach the equilibrium state in finite time. Such a property, called {\em finite-time stability} in 
\cite{APR,PR1,PR2} or {\em super-stability} in \cite{SLX}, was first noticed  in \cite{komornik,majda} for the (linear)
wave equation.  The extension of such a property to the wave equation on networks was addressed in \cite{APR,SLX}.  

Fortunately, the finite-time stability still occurs for systems of $2\times 2$  quasilinear hyperbolic equations of diagonal form without source terms, as it was noticed in 
\cite{LS} with initial data satisfying some compatibility conditions to prevent the emergence of shockwaves, and next in  \cite{PR1,PR2} for arbitrary
initial data by replacing homogeneous boundary conditions by some dynamical boundary conditions.  

The finite-time stabilization of a quasilinear hyperbolic system with source terms  seems to be very challenging. In \cite{CVKB}, the authors proved
that a  $2\times 2$  {\em linear} hyperbolic system with source terms can be stabilized to the origin in finite-time by using some boundary feedback laws designed with the backstepping
approach.   

On the other hand, the finite-time stability of a system may be lost when a small, bounded perturbation is added to the system. A famous example is provided 
by the {\em telegraph equation}
\begin{eqnarray}
\partial _t^2 y-\partial _x^2 y + \epsilon \partial _t y&=&0,\quad (t,x)\in (0,+\infty) \times (0,L), \label{i1}\\
y(t,0)=0, \quad y_x(t,L) &=&-y_t(t,L).  \label{i2} 
\end{eqnarray}
However, as it was noticed in \cite{gugat} for \eqref{i1}-\eqref{i2} or more generally for a nonlinear perturbation of the wave equation, the 
exponential stability of the system is preserved, with a decay rate proportional to $\ln (\epsilon ^{-1})$.  See also \cite{GLW} for 
the exponential stabilization of the isothermal Euler equations and \cite{DLP} for the loss of the stability when incorporating an arbitrarily  small delay in a transparent boundary condition for the wave equation.

The aim of this paper is to show that the robustness property noticed in  \cite{gugat}  is shared by most of the finite-time stable systems. The first result in this paper shows that a linear finite-time stable system with a (small) disturbance is exponentially stable with a decay rate proportional to the logarithm of the amplitude of the perturbation. We refer the reader to \cite[Theorem 4.2]{xu} for a sufficient condition involving the resolvent for the finite-time stability of a linear system. 

 \begin{theo}
 \label{thm1}
 Let $A$ be an operator generating a strongly continuous semigroup $(e^{tA})_{t\ge 0}$ in an Hilbert space $H$, and let $B\in {\mathcal L}(H)$ be a bounded operator. 
 Assume that $e^{TA}=0$ for some $T>0$. Then there exist some positive numbers  $\epsilon _0,M,C$ such that for any $\epsilon \in (0,\epsilon _0)$, it holds 
 \begin{equation}
 \label{D1}
 \Vert e^{ t (A+\epsilon B) } \Vert _{ {\mathcal L} (H) } \le M\inf (1, \frac{ e^{- (C\ln \epsilon ^{-1})t  }  }{\epsilon}  ) \quad  \forall t\ge 0.  
 \end{equation}
 \end{theo}

The (simple) proof of Theorem \ref{thm1} is given in Appendix. It rests on the observation that the solution 
of a Cauchy problem can be obtained as a fixed-point of a map, derived from Duhamel formula, in a weighted space.
The weight is related to the decay rate. 
 It is unclear whether such an approach could be extended to quasilinear systems.  
 
 It should be noticed that the estimate in \eqref{D1} is essentially sharp. Indeed, for the system 
 \begin{equation}
 \left\{
 \begin{array}{ll}
 \partial _t u + c\partial _x u = \epsilon v ,\quad &x\in (0,L),\ t>0 ,   \\
 \partial _t v - c\partial _x v = \epsilon u , \quad &x\in (0,L), \ t>0, \\
 u(t,0)=v(t,L)=0,\quad & t>0, \\
 u(0,x)=u_0(x), \quad v(0,x)=v_0(x), \quad &x\in (0,L), 
 \end{array}
 \right. 
 \end{equation} 
 we shall prove that the decay rate is roughly speaking bounded from below by $(c/L)\ln \epsilon ^{-1}$. 
  
  \begin{theo}
 \label{thm2}
 Let $c>0$, let $A(u,v):=(-c\partial _x u, c\partial _x v) $ be the operator with domain 
 \[
 D(A):=\{ (u,v) \in [ H^1(0,L) ]^2; \ u(0)=v(L)=0 \} \subset H := [L^2(0,L)]^2, 
 \]
 and let $B(u,v) := (v,u)$. 
 Then for any $\kappa >c/L$, there exist some numbers  $K,\epsilon _0>0$ such that for any $\epsilon \in (0,\epsilon _0)$, 
 it holds 
 \begin{equation}
 \label{DD1}
 \Vert e^{ t (A+\epsilon B) } \Vert _{ {\mathcal L} (H) }
 \ge  K e^{- (\kappa \ln \epsilon ^{-1})t  } \quad  \forall t\ge 0.  
 \end{equation}
 \end{theo}

The main aim of the paper is to investigate the application of transparent boundary conditions, or more generally of dynamical boundary conditions as in \cite{PR1,PR2}, to  $2\times 2$ quasilinear hyperbolic systems 
in diagonal form with {\em small} source terms
\begin{eqnarray}
\partial _t u + \lambda (u,v)\partial _x u &=&\epsilon f(u,v), \label{i4}\\  
\partial _t v - \mu (u,v)\partial _x v &=&\epsilon g(u,v) \label{i5} 
\end{eqnarray}
where $\lambda (u,v)>c$, $\mu (u,v) >c$ for some constant $c>0$, and $0\ <\epsilon\ll 1$.
Our results are stated when $f$ and $g$ do not depend on $\epsilon$, but there are still valid when $f$ and $g$ depend on $\epsilon$ but are 
 bounded in $W^{2,\infty}(0,L)$ for $0<\epsilon \ll 1$.  

We shall prove that for $\epsilon$ small enough and for initial data sufficiently close to a steady state of \eqref{i4}-\eqref{i5}, the solution 
of \eqref{i4}-\eqref{i5} with dynamical boundary conditions converge exponentially  to the steady state with a decay rate proportional to 
$\ln (\epsilon ^{-1})$.  

Our result can be applied to e.g. the {\em Saint-Venant system} with 
sources terms (see e.g. \cite{BMCPV,GL}), which is commonly used as a model for the water flow regulation in a canal with 
a slowly varying topography and some damping:
\begin{eqnarray}
\partial _t H + \partial _x (HV)&=&0, \label{i6}\\
\partial _t V + \partial _x (\frac{V^2}{2} + g H) &=&  - g \partial _x b -\frac{c_f}{2} \frac{V^2}{H}.  \label{i7}
\end{eqnarray}

In \eqref{i6}-\eqref{i7}, $t$ is time, $x$ is the space variable, $H=H(t,x)$ is the water depth, $V=V(t,x)$ is the flow velocity 
in the direction parallel to the bottom, $g$ is the gravitation constant, $c_f$ is the friction coefficient, and $z=b(x)$ is the equation of the bottom. 

Using the Riemann invariants 
\[
u:=V+2\sqrt{gH}, \quad v := V - 2\sqrt{gH}, 
\]
we easily see that system \eqref{i6}-\eqref{i7} can be rewritten as 
\begin{eqnarray}
\partial _t u + \lambda (u,v) \partial _x u &=&F(u,v), \label{i16}\\
\partial _t v -\mu (u,v) \partial _x v&=&F(u,v)  \label{i17}
\end{eqnarray}
where
\begin{eqnarray*}
&&\lambda  := V + \sqrt{gH} =  \frac{1}{4}(3u+v),\\
&&\mu  :=  -V + \sqrt{gH} = -\frac{1}{4} (u+3v), \\
&&F (u,v) := - g \partial _x b -2c_f g \left( \frac{u+v}{u-v} \right) ^2. 
\end{eqnarray*}
Thus, our results can be applied when $\Vert \partial _x  b\Vert _{L^\infty (0,L) }   + c_f \ll 1$. 

The above model is valid when the function $\partial _x b$ takes ``small  values''.  A more accurate model, the so-called {\em Savage-Hutter system} (see \cite{BMCPV}), reads
\begin{eqnarray}
\partial _t H + \partial _X (HV)&=&0, \label{I26}\\
\partial _t V + \partial _X (\frac{V^2}{2} + g \cos (\theta) H) &=& -g \sin ( \theta )   \cdot \label{i27}
\end{eqnarray}
Here, $X$ denotes a curvilinear coordinate along the bottom, $\theta =\theta (X)$ is the angle
of the bottom tangent with some fixed
horizontal axis,   $H=H(t,X)$ is the width of fluid in the normal direction at a point $X$ of the bottom, and $V(t,X)$ is the tangential velocity. 

Introducing the Riemann invariants
\begin{eqnarray*}
u & :=& V + 2\sqrt{g\cos (\theta )H}, \\
v &:=&  V - 2\sqrt{g\cos (\theta )H}, 
\end{eqnarray*}
we derive again a system of the form \eqref{i16}-\eqref{i17}, with $x=X$ and
\begin{eqnarray*}  
\lambda&:=& V + \sqrt{g\cos (\theta )H}  =\frac{1}{4} (3u+v),     \\
\mu       &:=& -V + \sqrt{g\cos (\theta )H}=-\frac{1}{4} (u+3v) ,        \\
F(u,v)    &:=& - g\sin (\theta ).
\end{eqnarray*}
Again, our results can be applied when $|\theta| \ll 1$. 

The paper is outlined as follows. The main result (Theorem \ref{theo:Principal}) is stated in Section 2. Its proof is displayed in Section 3.
It is divided in three parts. The first one is a rephrasing 
of the problem. The second part establishes the existence and uniqueness of global
solutions for small initial data using Schauder's fixed-point theorem. The last one introduces some Lyapunov functions 
with exponential weights needed to prove the exponential convergence towards the steady state.
The paper ends with an Appendix which contains the proofs of Theorem \ref{thm1} and of Theorem \ref{thm2} and which provides some background about linear transport equations.

\section{Stationary states and Main result}
\label{sec:Stationnaires}

We are interested in the following system of balance laws
\begin{equation}\label{eq:sysOriginal}
    \begin{cases}
        \partial_t u_\epsilon + \lambda(u_\epsilon, v_\epsilon) \partial_x u_\epsilon=\epsilon f(u_\epsilon, v_\epsilon)\\
        \partial_t v_\epsilon - \mu(u_\epsilon, v_\epsilon) \partial_x v_\epsilon=\epsilon g(u_\epsilon, v_\epsilon)
    \end{cases}
    t>0,\quad x\in (0,L).
\end{equation}

System \eqref{eq:sysOriginal} is supplemented with the {\em initial conditions}
\begin{equation}
u_\epsilon (0,x)=u_0(x), \quad v_\epsilon (0,x)=v_0(x)
\end{equation}
and the {\em boundary conditions} 
\begin{equation}\label{eq:conditionsBord}
            u_\epsilon (t,0)=y_l(t),\qquad v_\epsilon (t,L)=y_r(t).
\end{equation}
In \eqref{eq:conditionsBord}, the boundary data $y_l$ and $y_r$ are defined as the solutions of the 
initial value problem
\begin{equation}\label{eq:feedback}
    \begin{cases}
  \displaystyle  \frac{\mathrm{d}y_l}{\mathrm{d}t}=-K\displaystyle \frac{y_l-\bu}{|y_l-\bu|^\gamma}, \\[5mm]
   \displaystyle  \frac{\mathrm{d}y_r}{\mathrm{d}t}=-K\displaystyle \frac{y_r-\bv}{|y_r-\bv|^\gamma} ,\\
y_l(0) =u_0(0), \ \ y_r(0)=v_0(1), \end{cases}
\end{equation}
where $\gamma \in (0,1)$ and $K\in (0, + \infty)$ are any given numbers. 

Here and in what follows, we  assume that  the functions $\lambda,\ \mu,\ f$ and $g$ are of class  $\Cc^2$.

    We fix a pair  $(\bu,\bv)\in\R^2$ such that 
        \begin{equation}\label{eq:hypVitessePonctuelle}
\lambda(\bu,\bv)>0\qquad \mu(\bu,\bv)>0,
\end{equation}
and we introduce two real numbers $c>0$ and $R>0$ such that 
        \begin{equation}\label{eq:hypVitesse}
            \forall (u,v)\in \R^2,\qquad ||(u,v)-(\bu,\bv)||\leq 2R \Rightarrow \inf (\lambda(u,v),\mu(u,v))\geq c.
        \end{equation}


Let us construct for each $0<\epsilon \ll 1$ a pair $(\bu _\epsilon , \bv _\epsilon)$  of stationary states 
for \eqref{eq:sysOriginal}. 

    We denote by  $X$ the Banach space  of continuous functions from $[0,L]$ to $\R^2$ equipped with the
        uniform norm $\Vert (u,v) \Vert =\sup_{x\in [0,L]} || (u (x), v(x) ) ||$, and 
     by  $\Omega$ the subset of  $X$ consisting of the functions taking
       their  values in the open ball  $B((\bu,\bv),R)$; that is,
       \[
       \Omega = \{ (u,v) \in X; \ \Vert ( u(x) - \bar u, v(x) - \bar v) \Vert <R \quad \forall x\in [0,L]\}. 
       \] 
        
   We can now define a functional $F:(u,v,\epsilon)\in \Omega \times \R \to F(u,v, \epsilon )=(U,V) \in X$
       by
       \begin{eqnarray}
        U(x) &:=& u(x)-\bu-\epsilon \int_0^x{\frac{f(u(y),v(y))}{\lambda(u(y),v(y))}dy}, \label{P1}\\
        V(x)& :=& v(x)-\bv-\epsilon \int_x^L{\frac{g(u(y),v(y))}{\mu(u(y),v(y))}dy}. \label{P2}
\end{eqnarray}
    It is easy to see that 
        \begin{enumerate}
            \item[(i)] the set $\Omega$ is open;
            \item[(ii)] $F(\bu,\bv,0)=0$;
            \item[(iii)] the functional $F$ is of class $\Cc^1$ on $\Omega$;
            \item[(iv)] the differential with respect to $(u,v)$  of $F$ at $(\bar u, \bar v, 0)$  is given by 
            $\mathrm{D}_{(u,v)} F (\bu,\bv,0)=\mathrm{Id}_X$.
        \end{enumerate}

We infer from the implicit function theorem the local existence and uniqueness of a map
        $\epsilon \in [-\epsilon _0, \epsilon _0] \to (\bu_\epsilon,\bv_\epsilon)\in \Omega$ which is of class $\Cc^1$ and which satisfies 
\begin{equation}
\label{P3}
\forall \epsilon \in[-\epsilon_0,\epsilon_0],\qquad F(\bu_\epsilon,\bv_\epsilon)=0.
\end{equation}
    It follows then from   \eqref{P1}-\eqref{P2} that the  functions  $\bu_\epsilon$ and  $\bv_\epsilon$ are of class $\Cc^1$ in $[0,L]$
      and that they  satisfy
        \begin{equation}\label{eq:sysStationnaire}
            \begin{cases}
                \lambda(\bu_\epsilon,\bv_\epsilon)\partial_x \bu_\epsilon =\epsilon f(\bu_\epsilon, \bv_\epsilon), \quad \forall x\in [0,L],\\
                -\mu(\bu_\epsilon,\bv_\epsilon)\partial_x \bv_\epsilon =\epsilon g(\bu_\epsilon, \bv_\epsilon),\quad \forall x\in [0,L],\\
                \bu_\epsilon(0)=\bu,\\
                \bv_\epsilon(L)=\bv.
            \end{cases}
        \end{equation}

   We are now in a position to state the main result in this paper. 
\begin{theo}
\label{theo:Principal}
    There exist $\epsilon_0>0$ and $\delta>0$ such that for any $\epsilon \in [0,\epsilon_0]$,
     any $(\gamma , K) \in (0,1)\times (0,+\infty ) $,
     and any initial data $(u_0,v_0)\in \textrm{Lip} ( [0,L] )^2$, 
     with
    \begin{equation}
    \label{Q1}
        ||u_0-\bu||_{W^{1,\infty} (0,L)} \leq \delta \quad \textrm{ and } \quad  
        ||v_0-\bv||_{W^{1,\infty} (0,L)} \leq \delta,
  \end{equation}
    the system \eqref{eq:sysOriginal}-\eqref{eq:feedback} has a unique solution
    $(u_\epsilon, v_\epsilon)\in \Lip([0,+\infty)\times [0,L])$ satisfying \eqref{eq:sysOriginal}
    almost everywhere, and it holds
    \begin{eqnarray}
\forall t\geq 0,&& ||(u_\epsilon-\bu_\epsilon,
v_\epsilon-\bv_\epsilon)(t)||^2_{L^2 (0,L)} \leq
M  \inf ( 1, \frac{   e^{- C_\epsilon t }}  {\epsilon ^{1+\kappa}})
     \,    ||(u_0-\bu_\epsilon,v_0-\bu_\epsilon)||^2_{L^\infty (0,L)},\qquad\quad
      \label{eq:decroissanceL2}  \\
\forall t\geq 0,&& ||(u_\epsilon-\bu_\epsilon,
v_\epsilon-\bv_\epsilon)(t)||_{L^\infty(0,L)} \leq
M  \inf (1, \frac{   e^{- \frac{1}{3} C_\epsilon   t } } {\epsilon ^\frac{1+\kappa}{3}} )
    \,     ||(u_0-\bu_\epsilon,v_0-\bu_\epsilon)||^\frac{2}{3}_{L^\infty (0,L)},\qquad\quad 
 \label{eq:decroissance}
\end{eqnarray}
 where $\kappa := (c\delta ^\gamma)/(KL\gamma)$,  $M=M(\delta )>0$ and 
   $C_\epsilon=C_\epsilon (\delta) \sim   -\frac{c}{L}\ln(\epsilon)$ as  $\epsilon \to 0^+$.
\end{theo}

   We shall use some Lyapunov function  to prove \eqref{eq:decroissanceL2} (see 
   Section \ref{sec:Stabilisation}).\\
\begin{rema} 
 We notice  that for any fixed 
 \begin{equation}
 \label{extinction}
 t> (1+\kappa ) \frac{L}{c} = \frac{L}{c}  + \frac{\delta ^\gamma }{ K\gamma} , 
 \end{equation}
we have that 
\[
\frac{   e^ {- C_\epsilon  t } }{ \epsilon ^{1+\kappa}} = e^{(\ln \frac{1}{\epsilon} ) [1+\kappa -t (\frac{c}{L} +o(1)) ]} 
 \]     
and therefore that the r.h.s. of   \eqref{eq:decroissanceL2} tends to $0$ as $\epsilon \to 0^+$. This result, combined with the boundedness of $\{ (u_\epsilon, v_\epsilon) \} _{0<\epsilon<\epsilon _0}$ in $\textrm{Lip}([0,T]\times [0,L] ) $ for all $T>0$, 
yields again the finite-time stability around $(\bar u, \bar v)$ of the limit system (without source term) 
\[
\left\{ 
\begin{array}{ll}
\partial _t u + \lambda (u,v) \partial  _x u =0, \quad &t>0, \ x\in (0,L), \\
\partial _t v - \mu (u,v) \partial _x v=0, \quad &t>0, \ x\in (0,L),  \\
u(t,0)= y_l(t),\quad &t>0, \\
v(t,1)=y_r(t), \quad &t>0, \\
u(0,x)=u_0(x),\quad  &x\in (0,L), \\ 
v(0,x)=v_0(x)\quad &x\in (0,L). 
\end{array}
\right.
\]
that was established in \cite{PR2} with an extinction time very similar to \eqref{extinction} ($\delta ^\gamma$ in the r.h.s. of \eqref{extinction} being replaced by
${||(u_0-\bu ,v_0-\bv )|| }^\gamma_{L^\infty (0,L)}$).   
\end{rema} 

\section{Proof of Theorem \ref{theo:Principal}}
\subsection{Reduction of the problem}
\label{sec:Reduction}
We aim to show that if \eqref{Q1} holds with $\delta $ small enough, then the solutions of
 \eqref{eq:sysOriginal}-\eqref{eq:feedback} tend to
 $(\bu_\epsilon, \bv_\epsilon)$ as $t\to +\infty$. 

 To this end, we introduce the functions
$$
U:=u_\epsilon -\bar u_\epsilon, \qquad V:=v_\epsilon -\bar v_\epsilon.
$$

The original system \eqref{eq:sysOriginal} can be written

\begin{equation}\label{eq:sysPerturbe}
    \begin{cases}
        \partial_t U+\lambda(\bu_\epsilon+U, \bv_\epsilon+ V)\partial_x U=\epsilon  [ f(\bu_\epsilon +U,\bv_\epsilon+V)-f(\bu_\epsilon,\bv_\epsilon)] \\
       \qquad\qquad\qquad   -   
        \partial_x \bu_\epsilon [\lambda(\bu_\epsilon+ U, \bv_\epsilon+ V)-\lambda(\bu_\epsilon,\bv_\epsilon) ] ,\\
        \partial_t V-\mu(\bu_\epsilon+ U, \bv_\epsilon+ V)\partial_x V=\epsilon [g(\bu_\epsilon + U,\bv_\epsilon+ V)-g(\bu_\epsilon,\bv_\epsilon)] \\
        \qquad\qquad\qquad +
        \partial_x \bv_\epsilon  [\mu(\bu_\epsilon+ U, \bv_\epsilon+V)-\mu(\bu_\epsilon,\bv_\epsilon)] .\\
    \end{cases}
\end{equation}


The boundary conditions become
\begin{equation}\label{eq:bordPerturbe}
    \begin{cases}
     U (t,0)=y_l(t)-\bu,\\
     V (t,L)=y_r(t)-\bv, 
    \end{cases}
    \end{equation}
    where $y_l$ and $y_r$ still solve \eqref{eq:feedback}, and the initial condition read 
    \begin{equation}
    \label{AAA}
    U(0,x)=U_0(x):=u_0(x)-\bar u_\epsilon (x),\quad V(0,x)=V_0(x):= v_0(x)-\bar v_\epsilon (x). 
    \end{equation}
  
Note that $y_l(t)=\bar u$ (resp. $y_r(t)=\bar v$) for $t\ge  (\gamma K)^{-1} |U_0(0)|^\gamma$ (resp.  $t\ge (\gamma K)^{-1} |V_0(1)|^\gamma$). Therefore, 
if $\Vert U_0\Vert _{L^\infty (0,L)}\le \delta$ and $\Vert V_0\Vert _{L^\infty (0,L)}\le \delta$, then we have
\begin{equation}
\label{A10}
U(t,0)=V(t,1)=0 \quad \forall t\ge (\gamma K)^{-1} \delta ^\gamma.
\end{equation}

Thanks to the definition of $\Omega$, we notice that for any $\epsilon\in [-\epsilon_0,\epsilon_0]$, it holds 
$$
\forall x\in [0,L],\qquad ||(\bu_\epsilon(x),\bv_\epsilon(x))-(\bu,\bv)|| 
< R.$$

For any given $r>0$ and any $f:\R ^2\to \R$, let 
\[ \Vert f\Vert _r := \sup\{ \, \vert f(u,v)\vert; \  (u,v) \in B( (\bar u, \bar v); r)  \} .\]

In the following,  for any function $f=f(u,v)$, the quantity  $\Vert f\Vert $ will always denote $\Vert f\Vert _R$ where $R$ is as in \eqref{eq:hypVitesse}. (Note that 
$R$ does not depend on $\epsilon$.) For $h:[0,L]\to \R$, we denote $\Vert h\Vert _{[0,L]}=\sup\{ |h(x)|; \ x\in [0,L]\}$. 

Using \eqref{eq:sysStationnaire}, we see that
\begin{eqnarray*}
||\partial_x \bu_\epsilon||_{[0,L]} &\leq& \epsilon\frac{||f||}{c},\\
||\partial_x \bv_\epsilon||_{[0,L]} &\leq& \epsilon\frac{||g||}{c} \cdot
\end{eqnarray*}

Differentiating in \eqref{eq:sysStationnaire}, we infer that
    \begin{eqnarray*}
    ||\partial^2_{xx} \bu_\epsilon||_{[0,L]} &\leq& \epsilon^2\left(\frac{||f||.||\partial_1
            f||+||g||.||\partial_2 f||}{c^2}+
        \frac{||\partial_1 \lambda||.||f||^2}{c^3}+\frac{||\partial_2 \lambda||.||f||.||g||}{c^3}\right) ,\\
    ||\partial^2_{xx} \bv_\epsilon||_{[0,L]} &\leq& \epsilon^2\left(\frac{||f||.||\partial_1
            g||+||g||.||\partial_2 g||}{c^2}+ \frac{||\partial_2
            \mu||.||g||^2}{c^3}+\frac{||\partial_1 \mu||.||f||.||g||}{c^3}\right).
            \end{eqnarray*}

We now simplify the notations by setting
    \begin{equation}
        \begin{cases}
       \alpha(x,a,b):=\lambda(\bu_\epsilon(x)+a,\bv_\epsilon(x)+b),\\
       \beta(x,a,b):=\mu(\bu_\epsilon(x)+a,\bv_\epsilon(x)+b),\\
       F(x,a,b):=\epsilon [ f(\bu_\epsilon(x)+a,\bv_\epsilon(x)+b)-f(\bu_\epsilon(x),\bv_\epsilon(x)) ] \\
       \qquad\qquad\qquad -\partial_x \bu_\epsilon(x) 
       [ \lambda(\bu_\epsilon(x)+a,\bv_\epsilon(x)+b)-\lambda(\bu_\epsilon(x),\bv_\epsilon(x))]  ,\\
       G(x,a,b):=\epsilon [  g(\bu_\epsilon(x)+a,\bv_\epsilon(x)+b)-g(\bu_\epsilon(x),\bv_\epsilon(x)) ]\\
       \qquad \qquad \qquad + 
       \partial_x \bv_\epsilon(x)
       [ \mu(\bu_\epsilon(x)+a,\bv_\epsilon(x)+b)-\mu(\bu_\epsilon(x),\bv_\epsilon(x))] ,\\
       Y_l(t):=y_l(t)-\bu,\\
       Y_r(t):=y_r(t)-\bv. 
   \end{cases}  
   \end{equation}

The system  \eqref{eq:sysPerturbe} can be written as 
    \begin{equation}\label{eq:sysCanonique}
        \begin{cases}
            \partial_t U + \alpha(x,U,V) \partial_x U=F(x,U,V), \\
            \partial_t V - \beta(x,U,V) \partial_x V=G(x,U,V).
        \end{cases}
    \end{equation}
  It is supplemented with the initial conditions
  \begin{equation}
  \label{BBB}
 U(0,x)=U_0(x), \quad V(0,x)=V_0(x),  
\end{equation} 
and the boundary conditions
\begin{equation}\label{eq:bordCanonique}
    \begin{cases}
        U(t,0)=Y_l(t),\\
    V(t,L)=Y_r(t),
    \end{cases}
\end{equation}
where the functions $Y_l$ and $Y_r$ solve the system
\begin{equation}
\label{CCC}
    \begin{cases}
 \displaystyle     \frac{\mathrm{d}Y_l}{\mathrm{d}t}=-K\frac{Y_l}{|Y_l|^\gamma},\\[3mm]
  \displaystyle   \frac{\mathrm{d}Y_r}{\mathrm{d}t}=-K\frac{Y_r}{|Y_r|^\gamma},\\
    Y_l(0)=u_0(0)-\bar u, \quad Y_r(0)=v_0(1)-\bar v. 
    \end{cases}
\end{equation}

 The functions $\alpha$, $\beta$, $F$ and $G$ enjoy the following properties: 
    \begin{eqnarray} 
    \forall (a,b)\in B((0,0),R), \quad  &&\inf (\alpha(x,a,b),\beta(x,a,b)) \geq c ;  \label{Y1}\\
    \forall x\in [0,L],\quad   &&F(x,0,0)=0=G(x,0,0);\label{Y2} 
    \end{eqnarray}
    \begin{eqnarray*}
     &&\forall x\in[0,L],\ \forall (a,b)\in B((0,0),R) \\
      &&\qquad\qquad \qquad\qquad  \frac{|F(x,a,b)|}{||(a,b)||} \leq \epsilon\big( ||Df||_{B((\bu,\bv),2R)}+\frac{||f||_{B((\bu,\bv),R)}}{c}||D\lambda||_{B((\bu,\bv),2R)}\big) ;\\
       &&\qquad\qquad\qquad\qquad  \frac{|G(x,a,b)|}{||(a,b)||} \leq \epsilon\big( ||Dg||_{B((\bu,\bv),2R)}+\frac{||g||_{B((\bu,\bv),R)}}{c}||D\mu||_{B((\bu,\bv),2R)}\big) ;
\end{eqnarray*}
    \begin{eqnarray}
        &&\forall x\in[0,L],\quad \forall (a,b)\in B((0,0),R) \nonumber\\
        &&
        \begin{cases}
         |\partial_x F(x,a,b)|\leq \frac{2\epsilon^2}{c^2}\left(1+\frac{||\lambda||}{c}\right) \big(c||f||.||\partial_1
        f||+c||g||.||\partial_2 f||+||\partial_1 \lambda||.||f||^2+||\partial_2
        \lambda||.||f||.||g|| \big) ; \\[3mm]
       | \partial_x G (x,a,b)|\leq \frac{2\epsilon^2}{c^2}\left(1+\frac{||\mu||}{c}\right)
        \left(c||\partial_1 g||.||f||+c||\partial_2 g||.||g||+||g||^2.||\partial_2
        \mu||+||g||.||f||.||\partial_1 \mu|| \right)  ; \\[3mm]
\displaystyle     | \partial _a  
         F(x,a,b)|\leq \epsilon \big(||\partial_1 f|| _{2R} + \frac{||\partial_1 \lambda ||_{2R}  .||f||}{c}\big) ; \\[3mm]
\displaystyle     |\partial _b 
        F (x,a,b)|\leq \epsilon \big(||\partial_2 f|| _{2R} +\frac{||\partial_2 \lambda|| _{2R}.||f||}{c}\big) ; \\[3mm]
\displaystyle    |\partial _a  
        G(x,a,b)|\leq \epsilon \big(||\partial_1 g|| _{2R} +\frac{||\partial_1 \mu ||_{2R}  . || g 
         ||} {c}\big) ; \\[3mm]
 \displaystyle   | \partial _b 
        G(x,a,b)|\leq \epsilon \big(||\partial_2 g|| _{2R} +\frac{||   \partial _2    
        \mu||_{2R}  .|| g 
        ||}{c}\big) .
        \end{cases}
         \label{eq:borneGradient}
\end{eqnarray}

\subsection{The fixed-point argument}
\label{sec:PointFixe}
We show in this section the existence and uniqueness of the solution of \eqref{eq:sysPerturbe}-\eqref{AAA}
for $\epsilon$ small enough. Introduce some positive constants $I_0$,
$I_1$, $P$, and $M$ such that
$$||U_0||_{L^\infty(0,L)}\leq I_0,\qquad ||V_0||_{L^\infty(0,L)}\leq I_0,$$
$$||U'_0||_{L^\infty(0,L)}\leq I_1,\qquad ||V'_0||_{L^\infty(0,L)}\leq I_1,$$
$$\sup (||\partial_x F||_\infty,
||\partial _a  
F||_\infty,
|| \partial _b 
F||_\infty,||\partial_x
G||_\infty,||
\partial _a 
G||_\infty,||
\partial _b 
G||  _{\infty} 
)\leq P,$$
and
$$\sup (||\alpha||_\infty,||\partial_x \alpha||_\infty,
||\partial _a 
\alpha||_\infty,
|| \partial _b  
\alpha||_\infty,
||\beta||_\infty,||\partial_x \beta ||_\infty,
|| \partial _a 
\beta||_\infty,
|| \partial _b  
\beta || _{\infty} )\leq M.$$

We notice that we can assume that
\begin{enumerate}
    \item[(i)] the constants $I_0$, $I_1$ are small by picking  the numbers $\delta$ and $\epsilon _0$ in Theorem
        \ref{theo:Principal} small enough; 
    \item[(ii)] the constant $P$ is small by adjusting $\epsilon$ thanks to   \eqref{eq:borneGradient}; 
    \item[(iii)] the constant $M$ is bounded for $\epsilon \in [0,\epsilon _0]$.
\end{enumerate}

To display our fixed-point argument based on Schauder fixed-point theorem, we have to define (i) a class of functions containing the desired solution; (ii) a map whose
fixed-point is the desired solution of \eqref{eq:sysPerturbe}-\eqref{AAA}. This is done in the following two definitions.

\begin{defi}
    Given two positive constants $A$ and $B$, we define $\Dc$ as the set of functions
    $(U,V)\in \Lip( \R ^+ \times [0,L])^2$  such that
    \begin{equation}
        \forall (t,x)\in \R^+\times [0,L],\qquad 
        \begin{cases}
            |U(t,x)|\leq A,\\
            |V(t,x)|\leq A;\\
        \end{cases}
    \end{equation}
    \begin{equation}
        \forall (t,x,y)\in \R^+\times [0,L]^2,\qquad 
        \begin{cases}
            |U(t,x)-U(t,y)|\leq B|x-y|,\\
            |V(t,x)-V(t,y)|\leq B|x-y|;\\
        \end{cases}
    \end{equation}
    and
    \begin{equation}
        \forall (t,s,y)\in \R^+\times \R^+\times [0,L],\qquad 
        \begin{cases}
            |U(t,x)-U(s,x)|\leq B|t-s|,\\
            |V(t,x)-V(s,x)|\leq B|t-s|.\\
        \end{cases}
    \end{equation}
\end{defi}
\begin{defi} \label{def2}
    Given $(U,V)\in \Dc$, we introduce the unique solution  $(\tU,\tV)$  (thanks to Theorem
    \ref{theo:TransportSolution} in Appendix) of the system
    \begin{equation}
        \begin{cases}
            \partial_t \tU+\alpha (x,U,V)\partial_x \tU=F(x,U,V),\\
            \partial_t \tV-\beta(x,U,V)\partial_x \tV=G(x, U, V),\\
            \tU(t,0)=Y_l(t),\\
            \tV(t,L)=Y_r(t),\\
            \tU(0,x)=U_0(x),\\
            \tV(0,x)=V_0(x).
        \end{cases}
    \end{equation}
We set $\Fc (U,V) :=(\tU,\tV)$.
\end{defi}
The following result is classical (see e.g. \cite{R}). 
\begin{prop}
    The function $d: [ \Cc^0([0,+\infty[\times [0,L] )^2  ]^2 \to \R^+$ defined by
    $$
    d((U,V),(u,v)):=\sum_{n=1}^{+\infty}{\frac{1}{2^n}\frac{||U-u||_{L^\infty([0,n]\times[0,L])}+||V-v||_{L^\infty([0,n]\times[0,L])}}{1+||U-u||_{L^\infty([0,n]\times[0,L])}+||V-v||_{L^\infty([0,n]\times[0,L])}}}
    $$
    is a distance and  $ [ \Cc^0([0,+\infty[\times [0,L] )  ] ^2$  endowed with this distance is a Fr\'echet Space. On the other hand, $\Dc$ is a
    convex compact subset of $ [ \Cc^0([0,+\infty[\times [0,L] )  ] ^2$.
\end{prop}

We first have to show that $\Fc$ maps $\Dc$ into itself for an appropriate choice of the constants.  
\begin{prop} For any $A$, $B$, $c$, $K$, $L$, $M$ and $\gamma$ in $(0,+\infty )$, there exist some numbers $I_0,I_1$ and $P$ in $(0,+\infty )$  
 such that  $\Fc (\Dc ) \subset \Dc$. 
\end{prop}
\begin{proof}
    Using Corollary \ref{coro:estimeeTransport} in Appendix,  we obtain the following estimates: 
    \begin{equation}
        ||\tU||_{L^\infty}\leq \frac{2L}{c}P A +I_0;
    \end{equation}
    \begin{equation}
        ||\tV||_{L^\infty}\leq \frac{2L}{c}P A +I_0;
    \end{equation}
    \begin{equation}
        ||\partial_x \tU||_{L^\infty}\leq
        \frac{L}{c}\exp(\frac{L}{c}M(1+2B))\left(P(1+2B)+\frac{\sup(2PA+KI_0^{1-\gamma}, cI_1)}{L}\right) ;
    \end{equation}
    \begin{equation}
        ||\partial_x \tV||_{L^\infty}\leq
        \frac{L}{c}\exp(\frac{L}{c}M(1+2B))\left(P(1+2B)+\frac{\sup(2PA+KI_0^{1-\gamma}, cI_1)}{L}\right) ;
    \end{equation}
    \begin{equation}
        ||\partial_t \tU||_{L^\infty}\leq
        \frac{ML}{c}\exp(\frac{L}{c}M(1+2B))\left(P(1+2B)+\frac{\sup(2PA+KI_0^{1-\gamma},
                cI_1)}{L}\right)+2PA ;
    \end{equation}
    \begin{equation}
        ||\partial_t \tV||_{L^\infty}\leq
        \frac{ML}{c}\exp(\frac{L}{c}M(1+2B))\left(P(1+2B)+\frac{\sup(2PA+KI_0^{1-\gamma},
                cI_1)}{L}\right)+2PA.
    \end{equation}
    It is thus sufficient to show that for $I_0$ and $I_1$ small enough, one can choose $P$
    sufficiently small so that there exist some numbers  $A>0$ and $B>0$ with
\begin{eqnarray}
\frac{2L}{c}P A +I_0 &\leq&  A, \label{G1}\\
\frac{L}{c}\exp(\frac{L}{c}M(1+2B))\left(P(1+2B)+\frac{\sup(2PA+KI_0^{1-\gamma},
                cI_1)}{L}\right) &\leq&  B,\label{G2}\\
\frac{ML}{c}\exp(\frac{L}{c}M(1+2B))\left(P(1+2B)+\frac{\sup(2PA+KI_0^{1-\gamma},
        cI_1)}{L}\right)+2PA &\leq& B.\label{G3}
        \end{eqnarray}
For given $A,B,M,I_0$ and $I_1$ in $(0,+\infty)$, it is thus sufficient to impose that as $P\to 0^+$, the limit of the leftsided
terms in \eqref{G1}-\eqref{G3} are strictly less than those of the corresponding rightsided terms; that is, 
\begin{eqnarray}
I_0 &<& A, \label{Q11}\\
\frac{L}{c}\exp(\frac{L}{c}M(1+2B))\frac{\sup(K I_0^{1-\gamma}, cI_1)}{L} &<&  B, \label{Q12}\\
 \frac{ML}{c}\exp(\frac{L}{c}M(1+2B))\frac{\sup(KI_0^{1-\gamma},
  cI_1)}{L} &<&  B.\label{Q13}
 \end{eqnarray}
Now, it is clear that given $A$, $B$, $c$, $K$, $L$, $M$ and $\gamma$ in $(0,+\infty )$, 
 the conditions \eqref{Q11}-\eqref{Q13} are fulfilled by choosing $I_0$ and $I_1$ small enough. 
\end{proof}

Next, we have to show the continuity of  the map $\Fc$. 
\begin{prop}
    The map $\Fc: \Dc \to \Dc $ is continuous.
\end{prop}
\begin{proof}
    Let  $\{ (U_n,V_n)\} \subset  \Dc$ be any sequence such that 
    $$d((U_n,V_n),(U,V))\underset{n\to +\infty}{\to} 0$$
    for some $(U,V)\in \Dc$. 
    This is equivalent to saying that $(U_n)_{n\geq 0}$ tends to $U$ (and $(V_n)_{n\geq 0}$ tends
    to $V$) {\em uniformly} on all the sets $[0,T]\times [0,L]$, $T>0$. This yields that 
    $\alpha(.,U_n,V_n) \to \alpha(.,U,V)$ and  $\beta(.,U_n,V_n)\to \beta(.,U,V)$ uniformly on all
   the  sets $[0,T] \times [0,L]$, $T>0$.  Using Theorem \ref{theo:convTransport} in Appendix, 
    we deduce that $\Fc(U_n,V_n)$ converges uniformly on all the sets $[0,T]\times [0,L]$, $T>0$, to a pair
    $(\tilde{U},\tilde{V})\in \Dc$ which is the unique solution of the system
    \begin{equation}
        \begin{cases}
            \partial_t \tU+\alpha (x,U,V)\partial_x \tU=F(x,U,V),\\
            \partial_t \tV-\beta(x,U,V)\partial_x \tV=G(x, U, V),\\
            \tU(t,0)=Y_l(t),\\
            \tV(t,L)=Y_r(t),\\
            \tU(0,x)=U_0(x),\\
            \tV(0,x)=V_0(x).
        \end{cases}
    \end{equation}
 By Definition \ref{def2}, we conclude that $(\tilde U, \tilde V)=\Fc (U,V)$. It follows that
 $\Fc$ is continuous.
\end{proof}
We are now in a position to prove that $\Fc$ has a fixed-point in $\Dc$. 
\begin{prop}
    The map $\Fc$ has a fixed-point in $\Dc$, and therefore the system \eqref{eq:sysOriginal}-\eqref{eq:conditionsBord}
     has at least one solution.
\end{prop}
\begin{proof} As the map $\Fc:\Dc\to \Dc$ is continuous and $\Dc$ is a convex, compact subset of the Fr\'echet space $ [ \Cc^0([0,+\infty[\times [0,L] )  ] ^2$, it has at least one fixed-point by 
 Tihonov fixed-point theorem  (see e.g. \cite[Corollary 9.6]{zeidler}.) The pair $(U,V)$ solves  \eqref{eq:sysCanonique}-\eqref{eq:bordCanonique},
 while the pair $(u_\epsilon, v_\epsilon ) :=(U+\bar u_\epsilon, V+\bar v_\epsilon )$ solves 
 \eqref{eq:sysOriginal}-\eqref{eq:conditionsBord}.
\end{proof}

We are now concerned with the uniqueness of the solution of the system \eqref{eq:sysOriginal}-\eqref{eq:conditionsBord}.

\begin{prop}
    There is at most one solution to \eqref{eq:sysOriginal}-\eqref{eq:conditionsBord} in the class  $\textrm{Lip}([0,T]\times [0,L])^2$. 
\end{prop}
\begin{proof}
Clearly, it is sufficient to prove the uniqueness of the solution $(U,V)$ to the system  \eqref{eq:sysCanonique}-\eqref{eq:bordCanonique}.

Assume given  $U_0,V_0\in \textrm{Lip}([0,L])$ and two pairs $(U^i,V^i) \in \textrm{Lip}([0,T]\times [0,L])^2$ ($i=1,2$) of solutions of \eqref{eq:sysCanonique}-\eqref{eq:bordCanonique}; that is, we have for $i=1,2$
\[
\left\{
\begin{array}{l}
\partial _t U^i +\alpha (x, U^i,V^i) \partial _x U^i= F(x,U^i,V^i), \\
\partial _t V^i -\beta (x, U^i,V^i) \partial _x V^i= G(x,U^i,V^i), \\
U^i(t,0)=Y_l(t), \quad V^i(t,L)=Y_r(t), \\
U^i(0,x)=U_0(x), \quad V^i(0,x)=V_0(x). 
\end{array}
\right.
\]
Set $\hat U:=U^1-U^2$, $\hat V:=V^1-V^2$, and 
\begin{eqnarray*}
\hat \alpha &:=& \alpha  (x, U^1,V^1) - \alpha (x,U^2,V^2), \\
 \hat \beta &:=& \beta (x, U^1,V^1) - \beta (x,U^2,V^2), \\
 \hat F &:=& F(x, U^1,V^1) -  F(x,U^2,V^2), \\
 \hat G &:=& G(x, U^1,V^1) - G(x,U^2,V^2). 
\end{eqnarray*} 
Then we see that $\hat U, \hat V \in \textrm{Lip} ([0,T]\times [0,L])$ and that the pair $(\hat U,\hat V)$ solves the system
\begin{eqnarray}
\partial _t \hat U + \alpha (x, U^1,V^1) \partial _x \hat U + \hat \alpha  \partial _x U^2 &=& \hat F, \label{B1}\\ 
\partial _t \hat V -\beta (x,U^1,V^1)  \partial _x \hat V -\hat \beta  \partial _x V^2 &=& \hat G, \label{B2}\\
 \hat U (t,0)= \hat V(t,L)&=&0, \label{B3}\\
 \hat U (0,x)=\hat V(0,x) &=& 0.\label{B4}
\end{eqnarray} 
After multiplying \eqref{B1} (resp.  \eqref{B2}) by $\hat U$ (resp. by $\hat V$),  and integrating over $(0,t)\times (0,L)$, we obtain 
\begin{eqnarray}
\frac{1}{2} (\Vert \hat U\Vert ^2_{L^2(0,L)} + \Vert \hat V\Vert ^2 _{L^2(0,L)} )
&=& \int_0^t \!\!\!\int_0^L [\hat F\hat U+\hat G\hat V] dxd\tau  \nonumber \\
&&\quad   - \int_0^t\!\!\!\int_0^L [\alpha (x,U^1,V^1)\hat U\partial _x \hat U -\beta (x,U^1,V^1) \hat V \partial _x \hat V]dxd\tau\nonumber \\
&& \quad - \int_0^t\!\!\!\int_0^L [\hat \alpha \hat U \partial _x U^2 -\hat \beta \hat V \partial _x V^2 ]dxd\tau  =: I_1 -  I_2 - I_3.
\end{eqnarray}
Then 
\[
|I_1| \le \int_0^t \!\!\! \int_0^L  [ P(  |\hat U|+|\hat V|) |\hat U|   +P ( |\hat U |+ | \hat V| )|\hat V| ]   dxd\tau 
 \le 2P \int_0^t \!\!\! \int_0^L \big( |\hat U|^2 +|\hat V|^2\big) dxd \tau . 
\]
Similarly,  
\[
\vert I_3\vert \le MB \int_0^t\!\!\! \int_0^L   ( |\hat U| +|\hat V|) ^2 dxd\tau
\le 2MB\int_0^t\!\!\!\int_0^L (| \hat U | ^2 + | \hat V|^2) dxd\tau . 
\]
On the other hand, 
integrating by part in $I_2$ yields 
\begin{eqnarray*}
-I_2&=&\int_0^t\!\!\!\int_0^L ( \partial _x [\alpha (x,U^1,V^1)] |\hat U|^2 -\partial _x [\beta (x, U^1,V^1) ] |\hat V|^2 )dxd\tau \\
&&\quad -\int_0^t \left[ \alpha (x,U^1,V^1) |\hat U|^2 -\beta (x,U^1,V^1) |\hat V|^2 \right]_0^L d\tau\\
&\le& \int_0^t\!\!\!\int_0^L ( \partial _x [\alpha (x,U^1,V^1)] |\hat U|^2 -\partial _x [\beta (x, U^1,V^1) ] |\hat V|^2 )dxd\tau \\
&\le&  M(1+2B) \int_0^t\!\!\!\int_0^L ( | \hat U |^2 + | \hat V|^2)dxd\tau  
\end{eqnarray*}
where we used the definitions of $\alpha,\beta, M,B$, \eqref{eq:hypVitesse} and \eqref{B3}. 
Gathering together all the estimates, we arrive to
\[
\int_0^L [ |\hat U (t,x)|^2 + |\hat V (t,x)|^2] dx \le \left[ M(1+4B) +2P \right] \int_0^t\int_0^L ( |\hat U|^2 + |\hat V|^2)dxd\tau , \quad \forall t\ge 0.  
\] 
An application of Gronwall's lemma yields $\hat U\equiv 0$ and $\hat V\equiv 0$. 
\end{proof}

\subsection{Lyapunov Functions and Decay Rates}
\label{sec:Stabilisation}

Let $(U,V, Y_r,Y_l)$ denote the solution of system \eqref{eq:sysCanonique}-\eqref{CCC}.

To simplify the computations, we denote by $\tC$ a positive constant such that for all 
$x\in [0,L]$ and all $(a,b)\in B((0,0), R)$, we have
$$|F(x,a,b)|\leq \frac{\tC\epsilon}{2}(|a|+|b|),\qquad |G(x,a,b)|\leq \frac{\tC\epsilon}{2}(|a|+|b|)$$
and such that for all $(t,x)\in \R^+\times [0,L]$, we have 
$$\sup \big( \alpha(x,U(t,x),V(t,x)),\beta(x,U(t,x),V(t,x)) \big) \leq \tC$$
and 
   \begin{equation}
    \label{Y3}
    |\frac{d}{dx}\big( \alpha(x,U(t,x),V(t,x)) \big) |\leq \tC,\qquad
   |\frac{d}{dx}\big( \beta(x,U(t,x),V(t,x)) \big) |\leq \tC.
   \end{equation}

We first introduce a Lyapunov function to investigate the long-time behavior  of $(U,V)$. 
\begin{defi}
    Given any $\theta >0$, let the function $L_\theta$ be defined by
    $$ \Lc _\theta(a,b):=\int_0^L [ a^2(x)e^{-\theta
            x}+b^2(x)e^{-\theta(L-x)} ] dx \qquad \forall (a,b)\in [L^2(0,L)]^2.$$
\end{defi}
Then we have the following result. 
\begin{prop}
\label{prop1}
For almost every $t\in \R ^+$, it holds
$$\frac{d}{dt}\Lc _\theta(U(t,.),V(t,.))\leq(\tC\epsilon\frac{3+e^{\theta L}}{2}
+\tC-c\theta)\Lc _\theta(U(t,.),V(t,.))+\tC (Y_r^2(t)+Y_l^2(t)).$$
\end{prop}
\begin{proof}
    Since $U$ and $V$ are {\em uniformly bounded} and {Lipschitz continuous}, and hence differentiable almost everywhere by Rademacher theorem,
    the derivative of  $\Lc _\theta(U(t,.),V(t,.))$ exists almost everywhere and it is obtained by differentiating the integrand
    with respect to $t$. Thus we obtain
    \begin{align*}
\frac{d}{dt}\Lc _\theta(U(t,.),V(t,.))&=\int_0^L  2   {F(x,U(t,x),V(t,x))U(t,x)e^{-\theta x}dx}\\
&\quad +\int_0^L{  2  G(x,U(t,x),V(t,x))V(t,x)e^{-\theta (L-x)}dx}\\
&\quad -\int_0^L{\alpha(x,U(t,x),V(t,x))\partial_x (U^2)(t,x)e^{-\theta x}dx}\\
&\quad +\int_0^L{\beta(x,U(t,x),V(t,x))\partial_x (V^2)(t,x)e^{-\theta(L-x)}dx}\\
&=:I_1+I_2+I_3+I_4.
\end{align*}
Let us begin with $I_3$.
\begin{multline*}
  I_3=-[e^{-\theta x}\alpha(x,U(t,x),V(t,x))U^2(t,x)]_0^L-\theta
   \int_0^L{\alpha(x,U(t,x),V(t,x))U^2(t,x)e^{-\theta x}dx}\\
   +\int_0^L{\frac{d}{dx}(\alpha(x,U(t,x),V(t,x))) U^2(t,x)e^{-\theta x}dx}.
   \end{multline*}
As far  $I_4$ is concerned, we have that 
\begin{multline*}
  I_4=  
  [e^{-\theta(L- x)}\beta(x,U(t,x),V(t,x))V^2(t,x)]_0^L-\theta
   \int_0^L{\beta(x,U(t,x),V(t,x))V^2(t,x)e^{-\theta(L- x)}dx}\\
   -\int_0^\Lc {\frac{d}{dx}(\beta(x,U(t,x),V(t,x))) V^2(t,x)e^{-\theta(L- x)}dx}.
   \end{multline*}
   Using \eqref{Y1} and \eqref{Y3}, we infer that
   \begin{eqnarray}
    I_3+I_4 &\leq& \alpha(0,U(t,0),V(t,0))U^2(t,0)+\beta(L,U(t,L),V(t,L))V^2(t,L) \nonumber\\
   && \quad +(\tC-\theta c)
   \Lc _\theta(U(t,.),V(t,.)). \label{Y5}
   \end{eqnarray}
   Let us evaluate the remaining terms $I_1$ and $I_2$. We have that
   \begin{align}
   I_1+I_2&\leq \tC\epsilon\int_0^L{ [U^2(t,x)e^{-\theta
           x}+V^2(t,x)e^{-\theta(L-x)}+U(t,x)V(t,x)\left(e^{-\theta (L-x)}+e^{-\theta x}\right) ] dx} \nonumber\\
   &\leq \tC\epsilon
   \Lc _\theta(U(t,.),V(t,.))+\tC\epsilon\int_0^L{\frac{U^2(t,x)+V^2(t,x)}{2}(e^{-\theta(L-x)}+e^{-\theta x})dx}  \nonumber\\
   &\leq \tC \epsilon\frac{3+e^{\theta L}}{2}\Lc _\theta(U(t,.),V(t,.)), \label{Y6}
   \end{align}
   where we used twice in the last inequality the following estimate
   $$\forall x\in [0,L],\qquad e^{-\theta(L-x)}+e^{-\theta x}\leq (e^{\theta L}+1)e^{-\theta
       x}.$$
 The proof of Proposition \ref{prop1} is completed by gathering together \eqref{Y5} and \eqref{Y6}.  
\end{proof}
Next, we introduce another Lyapunov function for the dynamics of the boundary conditions. 
\begin{defi}
    Given any $\theta>0$, let  the function $\tL_\theta$ be defined by
    $$\forall (a,b)\in \R^2,\qquad \tL_\theta(a,b):= \frac{\tC |a|^{\gamma+2}}{K(\gamma+2)}e^{\theta \frac{c}{K\gamma}
        |a|^\gamma}+\frac{\tC|b|^{\gamma+2}}{K(\gamma+2)}e^{\theta \frac{c}{K\gamma} |b|^\gamma}.$$
\end{defi}

Then the following result holds.
\begin{prop}
\label{prop2}
    We have for all $t\ge 0$
  $$\frac{d}{dt}\tL _\theta (Y_l(t),Y_r(t))\leq-c\theta \tL _\theta (Y_l(t),Y_r(t))-\tC(Y_l^2(t)+Y_r^2(t)).$$ 
\end{prop}
\begin{proof}
    From
    $$\partial_a \tL _\theta (a,b)=\frac{\tC}{K(\gamma+2)}((\gamma+2)a |a|^\gamma+\frac{\theta c}{K}a
    |a|^{2\gamma})e^{\theta \frac{c}{K\gamma} |a|^\gamma} ,$$
    $$\partial_b \tL _\theta (a,b)=\frac{\tC}{K(\gamma+2)}((\gamma+2)b |b|^\gamma+\frac{\theta c}{K}b
    |b|^{2\gamma})e^{\theta \frac{c}{K\gamma} |b|^\gamma} ,$$
    and
    $$\dot{Y}_g(t)=-K\frac{Y_l(t)}{|Y_l(t)|^\gamma},$$
    $$\dot{Y}_d(t)=-K\frac{Y_r(t)}{|Y_r(t)|^\gamma},$$
    we infer that
    \begin{align*}
  \frac{d}{dt}\tL _\theta (Y_l(t),Y_r(t))
  &=\dot{Y}_g(t)\partial_a
  \tL _\theta (Y_l(t),Y_r(t))+\dot{Y}_d(t)\partial_b \tL _\theta(Y_l(t),Y_r(t))\\
  &=- \tilde C Y_l^2(t)e^{\theta \frac{c}{K\gamma} |Y_l(t)|^\gamma}- \tilde C Y_r^2(t)e^{\theta \frac{c}{K\gamma}
      |Y_r(t)|^\gamma}-c\theta \tL _\theta (Y_l(t),Y_r(t)) \\
      &\le - \tilde C Y_l^2(t)- \tilde C Y_r^2(t) -c\theta \tL _\theta (Y_l(t),Y_r(t)).
  \end{align*}
\end{proof}

We are in a position to define the Lyapunov function for the full state $(U,V,Y_l,Y_r)$. Let the functional $\Lc$ be defined by
 $$\forall t\geq 0,\qquad
   \Lc(U(t,.),V(t,.),Y_l(t),Y_r(t)):=\Lc _\theta(U(t,.),V(t,.))+\tL _\theta(Y_l(t),Y_r(t)).$$
Then the following holds.
\begin{theo}
    The functional $\Lc$ satisfies
   \begin{equation}
   \label{A1}
   \frac{d}{dt}\Lc\leq (\tC\epsilon\frac{3+e^{\theta L}}{2}
 +\tC-c\theta)\Lc\qquad \textrm{ for a.e. } t\ge 0. 
 \end{equation}
 On the other hand, the best decay rate obtained by taking the minimum of the parenthesis over $\theta$ reads
\begin{equation}
\label{A2}
-C_\epsilon := \underset{\theta \in \R}{\min} \left(\tC\epsilon\frac{3+e^{\theta L}}{2}
 +\tC-c\theta\right)
 = - \frac{c}{L}\ln (\epsilon ^{-1}) + [\frac{c}{L} +\tilde C -\frac{c}{L} \ln \frac{2c}{\tilde CL} ]+\frac{3\tilde C}{2}\epsilon\ 
 \underset{\epsilon \to 0^+}{\sim}\  -\frac{c}{L}\ln (\epsilon ^{-1}).
\end{equation}
\end{theo}
\begin{proof}
The estimate \eqref{A1} follows at once from Propositions \ref{prop1} and \ref{prop2}. On the other hand, the minimum of the function 
 $h(\theta ):= \tC\epsilon\frac{3+e^{\theta L}}{2}
 +\tC-c\theta$ is achieved when $e^{\theta L} =(2c)/(\tilde C\epsilon L)$, which yields \eqref{A2}.
\end{proof}

The estimate \eqref{A1} will give the exponential decay of the $L^2$ norm of $(U,V)$. To derive an exponential decay for the $L^\infty$ norm of
$(U,V)$, we need the following result.  
\begin{lemm}
\label{lem1}
    Let $u\in \textrm{Lip} ([0,L])$. Then
    \begin{equation}
    \label{L1}
    ||u||_{L^\infty (0,L)}\leq \frac{2}{\sqrt{L}}||u||_{L^2(0,L)}\quad \text{ or }\quad ||u||^3_{L^\infty(0,L)}\leq 8||u||_{L^2(0,L)}^2 ||\partial_x u||_{L^\infty(0,L)}.
    \end{equation}
    If, in addition $u(0)=0$, then 
    \begin{equation}
    \label{L2}
    ||u||^3_{L^\infty(0,L)}\leq 16 ||u||_{L^2(0,L)}^2 ||\partial_x u||_{L^\infty(0,L)} .
    \end{equation} 
\end{lemm}
\begin{proof}If $\Vert \partial _x u \Vert _{L^\infty (0,L)}=0$, then the function $u$ is constant and the first inequality in \eqref{L1} is obvious. Assume that $\Vert \partial _x u \Vert _{L^\infty (0,L)}>0$ and 
    pick any $x\in [0,L]$ such that 
    $$u(x)=||u||_{L^\infty(0,L)}.$$
    Then we have 
    $$\forall y\in [0,L],\qquad |u(y)|\geq |u(x)|-||\partial_x u||_{L^\infty(0,L)}|y-x|,$$
    Letting
    $$D:=\frac{|u(x)|}{2||\partial_x
        u||_{L^\infty(0,L)}}=\frac{||u||_{L^\infty(0,L)}}{2||\partial_x u||_{L^\infty(0,L)}},$$
    we see that 
    $$\forall y\in [0,L],\qquad |x-y|\leq D \Rightarrow |u(y)|\geq \frac{|u(x)|}{2}=\frac{||u||_{L^\infty(0,L)}}{2} \cdot$$
    Now if $I:=[\sup(0,x-D), \inf(L, x+D)]$, we have
    $$||u||_{L^2(0,L)}^2\geq \int_{I}{|u(y)|^2 dy}\geq |I| \frac{||u||_{L^\infty(0,L)}^2}{4},$$
    and hence
    $$||u||_{L^\infty(0,L)}^2\leq \frac{4||u||^2_{L^2(0,L)}}{|I|}\cdot $$
    From the definition of $I$ (using $x\in [0,L]$), we have that $|I|\geq \inf (L, D)$,  so that 
    $$||u||_{L^\infty(0,L)}^2\leq 4\sup \left(\frac{1}{D},\frac{1}{L}\right)||u||^2_{L^2(0,L)}\cdot $$
    Using the definition of $D$, we obtain \eqref{L1}. 
    
    Assume in addition that $u(0)=0$. We claim that 
    \begin{equation}
    \label{L3}
    \Vert u\Vert _{L^2(0,L)} ^2 \le 2L^\frac{3}{2} \Vert u \Vert_{L^2(0,L)}  \Vert \partial _x u\Vert_{L^\infty (0,L)}.
    \end{equation}
    Indeed, we have by Cauchy-Schwarz inequality 
    \[
    u(x)^2 -0  =\int_0^x 2u(y)\partial _x u(y)\, dy \le 2\Vert u\Vert _{L^2(0,L)} \Vert \partial _x u\Vert _{L^2(0,L)} \le 2\sqrt{L} \Vert u\Vert _{L^2(0,L)} \Vert \partial _x u\Vert _{L^\infty (0,L)} 
    \]
    and an integration w.r.t. $x\in (0,L)$ yields at once \eqref{L3}. 
    Next, if the first estimate in \eqref{L1} holds, taking the cube of each term and using \eqref{L3}, we arrive to \eqref{L2}. If the second estimate
    in \eqref{L1} holds, then \eqref{L2}  is obvious. 
\end{proof}


Let us complete the proof of Theorem \ref{theo:Principal}. Picking any $\theta \in (0,+\infty )$, we infer from \eqref{A1} that    \eqref{eq:decroissanceL2}
holds for $t\in [0, 1 + (1+\kappa ) L/c]$ for some constant $M>0$. Increasing $M$ if needed, we see that  \eqref{eq:decroissance} holds as well for $t\in [0, 1 + (1+\kappa ) L/c ]$, by using 
Lemma \ref{lem1}. Assume now that $t> 1 +  ( 1+\kappa ) L/c$. We pick $\theta:= L^{-1} \ln [ (2c)/(\tilde C \epsilon L)]$, so that, for $0<\epsilon<\epsilon _0$, we have
$h(\theta )=-C_\epsilon \sim -\frac{c}{L} \ln \epsilon ^{-1}$. 
It follows that 
\begin{eqnarray*}
\int_0^L [U^2(t,x)+V^2(t,x)]dx &\le& \int_0^L [e^{\theta (L-x)} U^2(t,x) + e^{\theta x} V^2(t,x)]dx  \\
&\le& e^{\theta L} {\mathcal L} (t) \\
&\le& e^{\theta L} e^{-C_\epsilon t} {\mathcal L} (0)  \\
&\le&\frac{2c}{\tilde C \epsilon L} e^{-C_\epsilon t} \left( \int_0^L  [ U^2_0(x) +V^2_0(x)] \, dx + \tL _\theta ( Y_l (0) , Y_r(0) ) \right) . 
\end{eqnarray*}
But we have 
\[
| Y_l(0) | =| u_0(0) - \bar u | = |u_0 (0)  -u_\epsilon (0)|\le \delta 
\]
and similarly $|Y_r(0)|\le \delta$, so that 
\begin{eqnarray*}
\tilde L_\theta ( Y_l (0) , Y_r(0) )  &\le&
 \frac{\tilde C}{K(\gamma +2)} \delta ^\gamma \exp( \theta\displaystyle \frac{c\delta ^\gamma}{K\gamma} ) 
\Vert (u_0-\bar u_\epsilon, v_0 - \bar v_\epsilon )\Vert ^2_{L^\infty (0,L)}  \\
&\le& 
\frac{\tilde C}{K(\gamma +2)} \delta ^\gamma \exp( \displaystyle \frac{c\delta ^\gamma}{KL\gamma}  \ln ( \frac{2c}{\tilde C \epsilon L} ) )
\Vert (u_0-\bar u_\epsilon, v_0-\bar v_\epsilon )\Vert ^2_{L^\infty (0,L) }  \\
&\le& \frac{Const}{\epsilon ^\kappa} 
\Vert (u_0-\bar u_\epsilon, v_0-\bar v_\epsilon )\Vert ^2_{L^\infty (0,L) } \cdot
\end{eqnarray*}
It follows that 
\[
\Vert ( u-u_\epsilon , v-v_\epsilon ) (t) \Vert ^2_{L^2(0,L)}   \le Const \frac{e^{- C_\epsilon t}}{\epsilon ^{1+\kappa}} 
\Vert (u_0-\bar u_\epsilon, v_0-\bar v_\epsilon )\Vert ^2_{L^\infty (0,L) } \cdot
\]
Thus the estimate  \eqref{eq:decroissanceL2} holds for $t> 1 + (1+\kappa ) L/c$. Finally,  \eqref{eq:decroissance} follows from  \eqref{eq:decroissanceL2} 
and \eqref{L1}. The proof of Theorem \ref{theo:Principal} is complete.

\section{Conclusion}
In this paper, we have shown the exponential decay of the
$L^2$-norm and of the $L^\infty$-norm
of solutions of a quasilinear hyperbolic system of balance laws with 
sufficiently
small source terms and appropriately chosen boundary controls. In fact,  
with vanishing source terms the decay rates become arbitrarily 
large, and thus in the limit, the case of finite-time stability is recovered. 
Since we work with solutions that are at each time $t\geq 0$ in
the space $W^{1,\infty}(0,\, L)$ with respect to the space variable,
the question arises whether also
the $L^\infty$-norm of the space derivative
of the solution decays exponentially.
This question will be the subject of future investigations.

\section{Appendix}
\subsection{Proof of Theorem \ref{thm1}}
We can find some constants $N\ge 0$ and $\omega\in \R$ such that $\Vert e^{tA}\Vert _{ {\mathcal L} (H) } \le Ne^{\omega t}$ for all $t\ge 0$.   
On the other hand, it is well known that $A+\epsilon B$ generates also a strongly continuous semigroup in $H$. For any $u_0\in H$, the (mild) 
solution of the Cauchy problem 
\begin{equation}
\label{UUU}
\partial _t u =Au+\epsilon Bu,\quad u(0)=u_0
\end{equation}
is denoted by $u(t)=e^{t(A+\epsilon B)}u_0$, and it is the solution in $C^0 ( \R ^+  , H)$ of the Duhamel integral equation
\begin{equation}
u(t)=e^{tA}u_0 + \int_0^t e^{ (t-s)A } [\epsilon B u(s) ]ds, \quad \forall t\ge 0.   
\end{equation}
Pick any $u_0\in H$. For given $\lambda \in (0,+\infty )$, introduce the Banach space
\[
E=\{ u\in C^0( \R ^+ , H); \ \Vert u\Vert _E:=\sup_{t\ge 0} \Vert e^{\lambda t} u(t)\Vert _H<+\infty \} .  
\]
For any $u\in E$, we define a function $v: \R ^+ \to H$ by 
\begin{equation}
\label{F1}
v(t) = e^{tA}u_0 + \int_0^t e^{ (t-s)A } [\epsilon B u(s) ]ds, \quad \forall t\ge 0.   
\end{equation}
Let $\Gamma (u)=v$. We aim to show that for $\epsilon \in (0,\epsilon _0)$ (with $\epsilon_0$ small enough) and $\lambda >0$ conveniently chosen, 
$\Gamma$ has a unique fixed-point in $E$.
First, we note that $\Gamma$ maps the space $E$ into itself. Indeed, for $t\in [0,T]$ we have 
\[
e^{\lambda t} \Vert v(t)\Vert \le e^{\lambda T} N \sup (1, e^{\omega T}) \left(  \Vert u_0\Vert + \epsilon _0 T  \Vert B\Vert _{ {\mathcal L} (H)} \Vert u\Vert _E   \right) <\infty,
\]
and for $t\ge T$, we have (using the fact that $e^{sA}w=0$ for all $s\ge T$ and all $w\in H$)
\begin{eqnarray*}
e^{\lambda t} \Vert v(t)\Vert 
&\le&  \int_{t-T}^t e^{\lambda t} \Vert e^{{(t-s)}A }\epsilon Bu(s)\Vert ds\\
&\le& N\sup (1, e^{\omega T})  \epsilon \Vert B\Vert  _{ {\mathcal L} (H)}  \Vert u\Vert _E \int_{t-T}^t e^{\lambda (t-s)}ds\\
&\le& N\sup (1, e^{\omega T})  \epsilon \Vert B\Vert  _{ {\mathcal L} (H)} \Vert u\Vert _E \frac{e^{\lambda T} -1}{\lambda}  <\infty. 
\end{eqnarray*}
Let us show now that $\Gamma$ contracts in $E$. Pick any $u_1,u_2\in E$, and let us denote $v_1=\Gamma (u_1)$, $v_2=\Gamma (u_2)$. Then we have 
for all $t\ge 0$ 
\begin{eqnarray*}
e^{\lambda t} \Vert v_1(t) -v_2(t) \Vert 
&\le&  \int_{\sup (0,t-T)}^t e^{\lambda t} \Vert e^{{(t-s)}A }\epsilon B (u_1(s) -u_2(s))\Vert ds\\
&\le& N\sup (1, e^{\omega T})  \epsilon \Vert B\Vert  _{ {\mathcal L} (H)}\frac{e^{\lambda T} -1}{\lambda}  \Vert u_1-u_2\Vert _E 
\end{eqnarray*}
 so that $\Gamma$ contracts for any given $\lambda >0$ if $\epsilon \ll 1$. Then by the contraction mapping theorem, $\Gamma $ has a unique fixed-point in $E$ which 
 is nothing but the mild solution of \eqref{UUU}. 
 
Now, pick $\lambda >0$ of the form $\lambda=C\ln \epsilon ^{-1}$. Then $\Gamma$ contracts in $E$ if 
 $k:= N\sup (1, e^{\omega T})  \epsilon \Vert B\Vert  _{ {\mathcal L} (H)} \frac{e^{\lambda T}}{\lambda} <1$, which becomes
\[
k=\epsilon ^{1-CT} \frac{N\sup (1, e^{\omega T}) \Vert B\Vert  _{ {\mathcal L} (H)}}{C\ln \epsilon ^{-1}} <1. 
\]  
This holds if $C<1/T$ and $\epsilon \in (0,\epsilon _0)$ with $\epsilon _0>0$ small enough. On the other hand, we have that 
\[
\Vert u(t) -e^{tA}u_0\Vert _E \le (1-k)^{-1} \Vert e^{tA}u_0-\Gamma (e^{tA}u_0)\Vert _E \le (1-k)^{-1}k\Vert u_0\Vert   
\]
ant that 
\[
\Vert e^{ tA } u_0\Vert _E \le N\sup (1, e^{\omega T}) e^{\lambda T}  \le Const \frac{\Vert u_0\Vert} {\epsilon}\cdot
\] 
It follows that $\Vert u\Vert _E \le Const\,  \Vert u_0\Vert /\epsilon $, i.e. $\Vert u(t)\Vert \le Const\,  e^{-(C\ln \epsilon ^{-1})t} 
\Vert u_0\Vert /\epsilon$ for all $t\ge 0$.  On the other hand, 
\[
\sup_{t\ge 0} \Vert u(t)\Vert \le \sup _{t\ge 0} \Vert u(t) -e^{tA} u_0\Vert + \sup_{t\ge 0} \Vert e^{tA} u_0\Vert 
\le [(1-k)^{-1} k + N \sup (1, e^{ \omega T} ) ] \Vert u_0\Vert . 
\] 
The proof of Theorem \ref{thm1} is complete.

\subsection{Proof of Theorem \ref{thm2}}
To prove \eqref{DD1}, it is sufficient to find an eigenvalue $\lambda$ of $A+\epsilon B$ of the form 
\[
\lambda \sim  -\frac{c}{L} \ln \epsilon ^{-1}. 
\]
Thus, we investigate the following spectral problem
\begin{eqnarray}
\lambda u + c \partial _x u &=& \epsilon v, \label{DDD1}\\
\lambda v - c \partial _x v &=& \epsilon u , \label{DDD2}\\ 
u(0)=v(L) &=& 0. \label{DDD222}
\end{eqnarray}
Differentiating with respect to $x$ in \eqref{DDD1}, replacing $\partial _x v$ by its expression in \eqref{DDD2} and next $v$ by its expression in \eqref{DDD1}, we arrive to the following ODE for $u$
\begin{equation}
\label{DDD3}
\partial _x ^2 u= \frac{\lambda ^2 -\epsilon ^2}{c^2} u. \\
\end{equation}
Similarly, we see that $v$ solves the ODE
\begin{equation}
\label{DDD4}
\partial _x ^2 v= \frac{\lambda ^2 -\epsilon ^2}{c^2} v. \\
\end{equation}
Let $\alpha\in \C$ be such that 
\begin{equation}
\label{DDD5}
\alpha ^2 = \frac{\lambda ^2 -\epsilon ^2}{c^2}\cdot 
\end{equation}
Then $u$ and $v$ can be written as 
\begin{equation}
\label{DDD6}
u(x)=Ae^{\alpha x} + Be^{-\alpha x}, \quad u(x)=Ce^{\alpha x} + De^{-\alpha x}
\end{equation}
for some constants $A,B,C,D\in \C$. 
Plugging the expressions of $u$ and $v$ in \eqref{DDD1}-\eqref{DDD2} yields the system
\[
\left\{ 
\begin{array}{ll}
\lambda (A e^{\alpha x} + B e^{-\alpha x} ) + c(A\alpha e^{\alpha x} -B\alpha e^{-\alpha x}) 
=\epsilon ( C e^{\alpha x} + D e^{-\alpha x}) ,\\
\lambda (C e^{\alpha x} + D e^{-\alpha x} ) - c(C \alpha  e^{\alpha x} -D\alpha e^{-\alpha x}) 
=\epsilon ( A e^{\alpha x} + B e^{-\alpha x}) ,\\
\end{array}
\right. 
\] 
which is satisfied if 
\begin{eqnarray}
(\lambda +c\alpha ) A &=& \epsilon C, \label{DDD11}\\
(\lambda -c\alpha ) B &=& \epsilon D, \label{DDD12}\\
(\lambda -c\alpha ) C &=& \epsilon A, \label{DDD13}\\
(\lambda +c\alpha ) D &=& \epsilon B.\label{DDD14}
\end{eqnarray}
Now, using \eqref{DDD5}, we see that \eqref{DDD12}  and \eqref{DDD13}  follow respectively from \eqref{DDD14} and
\eqref{DDD11},  so that \eqref{DDD11}-\eqref{DDD14} is equivalent to 
\[
C=\frac{\lambda + c\alpha} {\epsilon} A, \quad B=\frac{\lambda + c\alpha}{\epsilon} D. 
\] 
Setting 
\begin{equation}
\label{DDD21}
\mu = \frac{\lambda +c\alpha}{\epsilon}, 
\end{equation}
we wee that the solutions of \eqref{DDD1}-\eqref{DDD2} are the functions of the form
\[
u(x)=Ae^{\alpha x} + \mu De^{-\alpha x} , \quad  v(x)= \mu A e^{\alpha x} + D e^{-\alpha x},
\]
where $A,D\in \C$ are arbitrary. The boundary conditions \eqref{DDD222} yield the following constraints 
for $A$ and $D$:
\[
\left\{ 
\begin{array}{l}
A+\mu D=0, \\
\mu A e^{\alpha L} + D e^{-\alpha L}=0. 
\end{array}
\right.
\] 
The above system admits some nontrivial solutions if and only if the determinant of the associated matrix is null:\begin{equation}
\label{DDD22}
e^{-\alpha L} -\mu ^2 e^{\alpha L}=0. 
\end{equation}
Gathering together \eqref{DDD5}, \eqref{DDD21} and \eqref{DDD22}, we conclude that $\lambda$ is an eigenvalue
of the operator $A+\epsilon B$  if 
there is some $\alpha\in\C$ such that 
\[
\left\{
\begin{array}{l}
\lambda ^2 -c^2 \alpha ^2=\epsilon ^2, \\
\left( \frac{\lambda + c\alpha }{\epsilon }\right) ^2 =  e^{-2 \alpha L} . 
\end{array}
\right. ,
\] 
or equivalently if there are some numbers $\alpha\in\C$ and $s\in \{ -1, 1\}$ such that 
\begin{eqnarray}
\frac{\lambda + c\alpha}{\epsilon}\cdot \frac{\lambda -c\alpha}{\epsilon} &=&1, \label{DDD31}\\
\frac{\lambda +c\alpha}{\epsilon} &=& s e ^{-\alpha L}. \label{DDD32}
\end{eqnarray}
This yields (using $s^{-1}=s$) 
\begin{equation}
\frac{\lambda -c\alpha}{\epsilon} =s e^{\alpha L} \label{DDD33}.
\end{equation} 
Eliminating $\lambda$ in \eqref{DDD32}-\eqref{DDD33}), we conclude that  system \eqref{DDD31}-\eqref{DDD32} is equivalent to the system
\begin{eqnarray}
\frac{\lambda +c\alpha}{\epsilon} &=&s e^{-\alpha L} , \label{EEE1}\\
\frac{\sinh (\alpha L)}{\alpha L} &=& -\frac{cs}{L\epsilon} \cdot \label{EEE2}
\end{eqnarray}
To complete the proof, it is sufficient to pick $s=-1$ and to limit ourselves to the solutions $\alpha \in \R ^+$ of the equation 
\begin{equation}
\label{EEE3}
\frac{\sinh (L\alpha )}{L\alpha} =\frac{c}{L\epsilon} \cdot 
\end{equation}
It is easily seen that the map $x\mapsto \sinh x/x$ is (strictly) increasing on $(0,+\infty )$ and onto $(1,+\infty )$, so that for any
$\epsilon \in (0, c/L)$ there is a unique $\alpha = \alpha (\epsilon)  \in (0,+\infty)$ satisfying \eqref{EEE3}. Moreover, the map 
$\epsilon \to  \alpha $ is decreasing and $\alpha \to +\infty$ as $\epsilon \to 0^+$. Finally, 
$e^{L\alpha} /(2L\alpha) \sim c/(L\epsilon )$ yields by taking the logarithm $L\alpha  \sim \ln\,  \epsilon ^{-1}$. This gives
by using again \eqref{EEE3}
\[
L\alpha = \ln \epsilon ^{-1} + O(\ln (\ln \epsilon ^{-1}) ) 
\]
so that 
\[
\lambda = -c\alpha + o(\epsilon ) = -\frac{c}{L} \ln \epsilon ^{-1} + O(\ln (\ln \epsilon ^{-1}))
> -\kappa \ln \epsilon ^{-1}
\]
if $\kappa >c/L$ and $0<\epsilon \ll 1$. \qed

\subsection{Transport equation}
\label{sec:Transport}
This part follows closely the Appendix of \cite{P}.
\begin{defi}
   Let  $a=a(t,x)$ be a Lipschitz continuous function on $\R^+\times
    \R$ such that 
    \begin{equation}
    \label{AP1}
    \underset{(t,x)\in \R^+\times \R}{\sup} \frac{|a(t,x)|}{1+|x|}<+\infty.
    \end{equation}
    For any pair $(t,x)\in \R^+\times \R$, we  denote by  
    $s \mapsto \phi_a(s,t,x)$
    the maximal solution to the following Cauchy problem
    $$
    \begin{cases}
        \dot{\theta}(s)=a(s,\theta(s)),\\
        \theta(t)=x,
    \end{cases}
    $$
   which is defined on $\R^+$ thanks to \eqref{AP1}. 
\end{defi}

\begin{prop}\label{prop:flot}
Assume that $a=a(t,x)$ satisfies \eqref{AP1} and is of class  $\Cc^1$ on $\R^+\times\R$.   Then the function $\phi_a$ is  of class $\Cc^1$ on $(\R ^+ )^2\times \R$, 
 and we have
   $$\forall (s,t,x)\in (\R^+)^2\times \R,\quad  
       \begin{cases}
           \partial_2 \phi(s,t,x)=-a(t,x)\exp\left(\int_t^s{\partial_2 a(r,\phi(r,t,x))dr}\right) ,\\
           \partial_3 \phi(s,t,x)=\exp\left(\int_t^s{\partial_2 a(r,\phi(r,t,x))dr}\right).
       \end{cases}
       $$
\end{prop}
\begin{proof}
    Using the integral form of the differential equation, we obtain
    \begin{equation}
    \label{PP1}
    \forall (s,t,x)\in (\R^+) ^2\times \R,\qquad \phi(s,t,x)=x+\int_t^s{a(r,\phi(r,t,x))dr}.
    \end{equation}
    We infer from the implicit function theorem (proceeding as in Section \ref{sec:Stationnaires}) that
    $\phi$ is of class $\Cc^1$ on $(\R ^+)^2\times \R$. On the other hand, differentiating in \eqref{PP1}  yields
    $$\forall (s,t,x)\in (\R^+)^2\times \R,\qquad \partial_2 \phi(s,t,x)=-a(t,\phi(t,t,x))  +   \int_t^s{\partial_2
        a(r,\phi(r,t,x))\partial_2 \phi(r,t,x) dr}.$$
        
    Noticing that the last equation can be viewed as a linear ODE and using the fact that $a(t,\phi(t,t,x))=a(t,x)$, we obtain the
    desired formula for $\partial _2 \phi$. The other one for $\partial _3\phi$ is proven in the same way.
\end{proof}
 
\begin{rema}
    Is should be noted that
    $$\forall (s,t,x)\in (\R^+)^2\times \R,\qquad \partial_t\phi(s,t,x)+a(t,x)\partial_x \phi(s,t,x)=0.$$
\end{rema}
\begin{defi}
    Given a positive real number  $L$, we introduce for any  $(t,x)\in \R^+\times [0,L]$ the sets 
    \begin{eqnarray*}
    F^- &:=& \{r\in [0,t]; \ \forall s\in [r,t],\quad \phi(s,t,x)\in [0,L]\}, \\
    F^+ &:=& \{r\in [t,+\infty ); \ \forall s\in [t,r],\quad \phi(s,t,x)\in [0,L]\}.
    \end{eqnarray*}
    We also set
   $$e(t,x):=\inf F^-. 
 $$
\end{defi}
\begin{prop}\label{prop:temps}
    Assume that $a$ is a function of class  $\Cc^1$ and that there exists $c>0$ such that
    \begin{equation}
    \label{PP2}
    \forall (t,x)\in [0,+\infty )  \times [0,L],\qquad a(t,x)\geq c. 
    \end{equation}
    Then the function $e$ is of class  $\Cc^1$ on the two open sets 
    \begin{eqnarray*}
    G &:=& \{(t,x)\in (0,T)\times (0,L); \ x<\phi(t,0,0)\} , \\
    I &:=& \{(t,x)\in (0,T)\times (0,L); \ x>\phi(t,0,0)\}.
    \end{eqnarray*} 
    Furthermore, it holds
    $$
    \forall (t,x)\in G,\qquad 
    \begin{cases}
        \partial_t e(t,x)=  -   \frac{a(t,x)}{a(e(t,x),0)}\exp \left(-\int_{e(t,x)}^t{\partial_2 a(r,\phi(r,t,x))dr}\right) ,\\
        \partial_x e(t,x)=-\frac{1}{a(e(t,x),0)}\exp \left(-\int_{e(t,x)}^t{\partial_2 a(r,\phi(r,t,x))dr}\right) ,
    \end{cases}
    $$
    while
     $$\forall (t,x)\in I,\qquad e(t,x)=0.$$
    The set $G$ coincide with the set of pairs $(t,x)\in (0,T)\times (0,L)$ such that 
    $$e(t,x)>0.$$
\end{prop}
\begin{rema}
    The following holds
    $$\forall (t,x)\in (0,T)\times (0,L) ,\qquad  x\neq \phi(t,0,0)\Rightarrow \partial_t e(t,x)+a(t,x)\partial_x e(t,x)=0.$$
    Since the propagation speed $a(t,x)$ is greater than $c$, it should be clear that
    $$\forall (t,x)\in \R^+\times [0,L],\qquad 0\leq t-e(t,x)\leq \frac{L}{c} \cdot$$
\end{rema}
\begin{proof}
    Using \eqref{PP2}, 
    we see that for $(t,x)\in G$, $e(t,x)$ is the only solution (at least locally) of the equation
    \begin{equation}
    \label{PP3}
    \phi(e(t,x),t,x)=0.
    \end{equation}
    An application of the implicit function theorem gives  that the function $e$ is of class $\Cc^1$.  Taking partial derivatives in \eqref{PP3} yields
    $$\partial_t e(t,x)=-\frac{\partial_2
        \phi(e(t,x),t,x)}{\partial_1 \phi (e(t,x),t,x)}=
        - \frac{-a(t,x)\exp\left(\int_t^{e(t,x)}{\partial_2
            a(r,\phi(r,t,x))dr}\right)}{a(e(t,x),0)}$$
            and 
    $$\partial_x e(t,x)=-\frac{\partial_3 \phi(e(t,x),t,x)}{\partial_1
        \phi(e(t,x),t,x)}=-\frac{\exp\left(\int_t^{e(t,x)}{\partial_2
                a(r,\phi(r,t,x))dr}\right)}{a(e(t,x),0)} \cdot$$
\end{proof}
Let us now consider the system
\begin{equation}\label{eq:transport}
    \begin{cases}
        \partial_t y+a(t,x)\partial_x y=b(t,x),\\
        y(t,0)=y_l(t),\\
        y(0,x)=y_0(x)
    \end{cases}\qquad (t,x)\in (0,T)\times (0,L).
\end{equation}
We say that a function $y\in \textrm{Lip} ( [0,T]\times [0,L])$   is a {\em strong solution} of \eqref{eq:transport} if the first equation in \eqref{eq:transport} holds almost everywhere and if
the second and third equations in \eqref{eq:transport} hold everywhere. 
We shall need also to introduce the concept of {\em weak solution} of \eqref{eq:transport}, following \cite{P}. 
\begin{defi}
\label{defi7}
    We say that $y\in \Cc ^0 ([0,T]\times [0,L])$ is a {\em weak solution} of \eqref{eq:transport} if
  for any function $\psi\in \Cc^1([0,T]\times [0,L])$ satisfying
    \begin{eqnarray*}
    \forall t \in [0,T],\quad \psi(t,L) &=& 0,\\
     \forall x\in [0,L],\quad \psi(T,x) &=&0,
     \end{eqnarray*}
    it holds
    \begin{multline}
    \int_0^T{\int_0^L{ [y(t,x)\bigg(  \partial_t \psi(t,x) + a(t,x)\partial_x
                \psi(t,x)+\partial_x a(t,x)\psi(t,x)\bigg) +b(t,x)
            \psi(t,x)  ] dx} dt}\\
    +\int_0^T   { a(t,0)        y_l(t)\psi(t,0)dt}+\int_0^L{y_0(x)\psi(0,x)dx}=0.
\end{multline}
\end{defi}
Then we have the following result. 
\begin{prop} 
\label{prop10}
1. Let $y\in \textrm{Lip} ([0,T]\times [0,L])$. Then $y$ is a strong solution of \eqref{eq:transport} if and only if $y$ is a weak solution of 
\eqref{eq:transport}.\\
2. If $a$, $b$, $y_l$ and $y_0$ are 
            Lipschitz continuous functions, then 
there is at most one weak solution of \eqref{eq:transport}.
\end{prop}
\begin{proof}
    The first assertion follows from classical arguments. The second one is proven in the
    Appendix of \cite{P}.
\end{proof}

\begin{theo}\label{theo:TransportSolutionRegu}
    Let  $a$, $b$, $y_l$ and $y_0$ be functions of class $\Cc ^1 $ such that 
    \begin{equation}\label{eq:compatibilite}
    y_l(0)=y_0(0).
    \end{equation}
    Then  the system \eqref{eq:transport} admits exactly one (strong or weak) solution, and it is given explicitly by the formula
    \begin{equation}\label{eq:formule}
        \forall (t,x)\in [0,T] \times [0,L],\qquad 
        y(t,x)=
        \left\{
        \begin{array}{ll}
       y_l(e(t,x))+\int_{e(t,x)}^t{b(r,\phi(r,t,x))dr} &\textrm{if } x<\phi(t,0,0),\\
       y_0(\phi(0,t,x))+\int_0^t{b(r,\phi(r,t,x))dr}  &\textrm{otherwise.}
        \end{array}
        \right.
    \end{equation}
\end{theo}
\begin{proof}
    Using Propositions \ref{prop:flot} and \ref{prop:temps}, it is straightforward to check
    that the function $y$ given by formula \eqref{eq:formule} is in $\textrm{Lip}([0,T]\times [0,L])$ and is of class $\Cc^1$, except possibly on the curve $x=\phi(t,0,0)$ (where
    it is likely merely continuous), and that it is a strong solution of \eqref{eq:transport}. On the other hand, the uniqueness of a weak
    solution of \eqref{eq:transport} follows from Proposition \ref{prop10}.  
\end{proof}
\begin{coro}\label{coro:estimeeTransport}
    The solution $y$ of \eqref{eq:transport} satisfies the estimates:
\begin{equation}\label{eq:estimationI}
    ||y||_\infty \leq \frac{L}{c} ||b||_\infty +\sup\left(||y_0||_\infty,||y_l||_\infty\right),
\end{equation}    
\begin{equation}\label{eq:estimationII}
   ||\partial_x y||_\infty \leq \frac{1}{c} \exp\left(\frac{L}{c}||\partial_2 a||_\infty\right) \sup \left(||b||_\infty
   + ||   y_l' 
   ||_\infty+L||\partial_2 b||_\infty,c||y_0'||_\infty+L||\partial_2 b||_\infty\right) ,
\end{equation}
and 
\begin{equation}\label{eq:estimationIII}
   ||\partial_t y||_\infty \leq  \frac{||a||_\infty}{c} \exp\left(\frac{L}{c}||\partial_2 a||_\infty\right) \sup \left(||b||_\infty+
   ||
    y_l' 
   ||_\infty+L||\partial_2 b||_\infty,c||y_0'||_\infty+L||\partial_2 b||_\infty\right)+||b||_\infty.
\end{equation}
\end{coro}
\begin{proof}
    Straightforward from \eqref{eq:formule}.
\end{proof}

\begin{prop}
 Let $(t,x)\in[0,T]\times [0,L]$ and  let $\{ a_n\} \subset  \mathcal{C}^0([0,T]\times [0,L])\cap L^\infty(0,T,\Lip([0,1]))$ be  a sequence such that 
 $||a_n||_{L^\infty(0,T 
 ,\Lip([0, L 
 ]))}$  is bounded and  
 $$||a_n-a||_{\mathcal{C}^0([0,T]\times [0,L])}\rightarrow 0  \textrm{ as } n\to +\infty,  $$
 and let $\{ (t_n,x_n) \} \subset  [0,T]\times [0,L]$ be a sequence such that  $(t_n,x_n)\rightarrow (t,x)$. Then 
$$e_n(t_n,x_n)\rightarrow e(t,x). $$
\end{prop}

\begin{proof}
    See the Appendix in \cite{P}.
\end{proof}
\begin{theo}\label{theo:convTransport}
    Let $y_n$,  $a_n$,  $b_n$, $y_{0,n}$ and $y_{l,n}$ be Lipschitz continuous functions such
    that 
    \begin{enumerate}
        \item[(i)] for any $n\geq 0$, the function  $y_n$ is a strong solution of
    \begin{equation}
    \label{AAA1}
    \begin{cases}
        \partial_t y_n+a_n(t,x)\partial_x y_n=b_n(t,x),\\
        y_n(t,0)=y_{l,n}(t),\\
        y_n(0,x)=y_{0,n}(x);
    \end{cases}\qquad (t,x)\in (0,T)\times (0,L)
    \end{equation}
\item[(ii)] the sequence $(a_n)_{n\geq 0}$ is bounded in $\Lip([0,T]\times [0,L])$ and it converges uniformly on $[0,T]\times [0,L]$ towards
    a function $a$;
\item[(iii)] the sequence $(b_n)_{n\geq 0}$ is bounded in $\Lip([0,T]\times [0,L])$ and  it converges uniformly on $[0,T]\times [0,L]$ towards
    a function $b$;
\item[(iv)] the sequence $(y_{l,n})_{n\geq 0}$ is bounded in $\Lip([0,T])$ and it converges uniformly on $[0,T]$ towards a
    function $y_l$;
\item[(v)] the sequence $(y_{0,n})_{n\geq 0}$ is bounded in $\Lip([0,T])$ and it converges uniformly on $[0,T]$ towards a
    function $y_0$.
\end{enumerate}
Then there exists a Lipschitz continuous  function $y$ on $[0,T]\times [0,L]$ such that
    $$y_n \to y \text{ in }\Cc^0([0,T]\times [0,L]),$$ 
    and $y$ is the unique solution of
    \begin{equation}
    \label{AAA2}
    \begin{cases}
        \partial_t y+a(t,x)\partial_x y=b(t,x),\\
        y(t,0)=y_l(t),\\
        y(0,x)=y_0(x).
    \end{cases}\qquad (t,x)\in (0,T)\times (0,L).
    \end{equation}
\end{theo}
\begin{proof}
Since $y_n$ is the unique solution of \eqref{AAA1}, it follows from Corollary
\ref{coro:estimeeTransport} that the sequence $(y_n)_{n\geq 0}$ is bounded in
$\Lip([0,T]\times [0,L])$. From Ascoli-Arzela theorem, we can extract
a subsequence $(y_{n_k})_{k\geq 0}$ which converges uniformly towards a function $y$. On the other hand, 
since $(\partial_x a_n)_{n\geq 0}$ is bounded in $L^\infty((0,T)\times (0,L))$ and since the
only possible weak$-*$ limit of a convergent subsequence is $\partial_x a$, we have by weak$-*$ compactness that 
$\partial_x a_n\to \partial_x a$ weakly$-*$ in $L^\infty$. But this is enough to pass to the limit as $k\to +\infty$ 
in the weak formulation 
    \begin{multline}
    \int_0^T{\int_0^L [{y_{n_k}(t,x)\left(\partial_t \psi(t,x) + a_{n_k}(t,x)\partial_x
                \psi(t,x)+\partial_x a_{n_k}(t,x)\psi(t,x)\right) +b_{n_k}(t,x)
            \psi(t,x)]dx}dt}\\
    +\int_0^T{  a_{n_k}(t,0) y_{l,n_k}(t)\psi(t,0)dt}+\int_0^L{y_{0,n_k}(x)\psi(0,x)dx}=0,
\end{multline}
where $\psi$ is any function as in Definition \ref{defi7}. 
We arrive to
\begin{multline}
    \int_0^T{\int_0^L{[y(t,x)\left(\partial_t \psi(t,x) + a(t,x)\partial_x
                \psi(t,x)+\partial_x a(t,x)\psi(t,x)\right) +b(t,x)
            \psi(t,x) ] dx}dt}\\
    +\int_0^T{ a(t,0) y_l(t)\psi(t,0)dt}+\int_0^L{y_0(x)\psi(0,x)dx}=0,
\end{multline}
that is,  $y$ is a weak solution of the transport equation \eqref{AAA2}.  As the weak solution of \eqref{AAA2} is unique by Proposition \ref{prop10}, we 
infer that  there is only one possible limit for any convergent subsequence of $(y_n)_{n\geq 0}$, so that the whole sequence $(y_n)_{n\ge 0}$ converges uniformly towards $y$. 
\end{proof}

\begin{theo}\label{theo:TransportSolution}
    Theorem \ref{theo:TransportSolutionRegu} and Corollary \ref{coro:estimeeTransport}
 are still valid when we assume merely that the functions  $a$, $b$, $y_l$ and $y_0$ are Lipschitz continuous.
\end{theo}

\section*{Acknowledgments}
VP and LR were partially supported by the ANR project Finite4SoS (ANR-15-CE23-0007).

\end{document}